\documentclass{myclass}

\newcommand\tod{\xrightarrow{\;d\;}}

\newcommand\U{\mathrm u}
\newcommand\D{\mathrm d}
\newcommand\ath{$a$\nobreakdashes-th}

\DeclareMathOperator\geom{geom}
\DeclareMathOperator\poisson{Poisson}
\renewcommand\abs[1]{\lvert#1\rvert}
\algrenewcommand\algorithmicrequire{\textbf{Global variable:}}

\usepackage{figures}

\usepackage[hidelinks]{hyperref}

\title{Progressive and rushed Dyck paths}
\author{Axel Bacher}
\date{May 21, 2026}

\begin{document}
\maketitle

\begin{abstract}
\noindent%
We call \emph{progressive paths} and \emph{rushed paths} two families of Dyck
paths studied by Asinowski and Jelínek, which have the same enumerating
sequence (OEIS entry A287709). We present a bijection proving this fact.
Rushed paths turn out to be in bijection with \emph{one-sided trees},
introduced by Durhuus and Ünel, which have an asymptotic enumeration involving
a stretched exponential. We show some asymptotic results and present a random
sampling procedure for rushed paths. We conclude by presenting several other
classes of related lattice paths and directed animals that may have similar
asymptotic properties.
\end{abstract}

\section{Introduction} \label{sec:intro}

Dyck paths---sequences of up and down steps starting and ending at~$0$
and staying at nonnegative height---are probably the most famous family of
lattice paths, one of the many combinatorial classes enumerated by the Catalan
numbers (OEIS entry \oeis{A000108}). The aim of this article is to study two
families of Dyck paths defined in the OEIS entry \oeis{A287709} and studied by
Asinowski in relation to certain rectangulation models \cite{asinowski}. We
define these paths below and illustrate them in Figure~\ref{fig:defs}.
Throughout the paper, we call \emph{height} of the path~$P$, and we denote
by~$h(P)$, the maximum height visited by~$P$.

\begin{definition} \label{def}
Let~$P$ be a Dyck path of height~$h$. We say that~$P$ is \emph{progressive}
if, for~$i = 2,\dotsc,h$, it visits the height~$i-1$ at least twice before the
first visit at height~$i$. We say that~$P$ is \emph{rushed} if it starts
with~$h$ up steps and then never again visits the height~$h$.
\end{definition}

\begin{figure}[ht]\hfill%
\begin{tikzpicture}[paths]
\foreach \y in {0,...,4} { \draw[help lines] (0,\y) -- ++(26,0); }
\draw[path] (0,0)
\walk u; \point{color1}
\walk du; \point{color2}
\walk u; \point{color1}
\walk du; \point{color2}
\walk duu; \point{color1}
\walk ddduduuu; \point{color2}
\walk u; \point{color1}
\walk dudddudd;;
\end{tikzpicture}\hfill\hfill%
\begin{tikzpicture}[paths]
\foreach \y in {0,...,4} { \draw[help lines] (0,\y) -- ++(26,0); }
\draw[path] (0,0) \walk uuuudduddududduduuududdudd;;
\end{tikzpicture}\hfill\hspace{0pt}%
\caption{Left: a progressive path of height~$4$ with the first and second
visits at each height marked with a red and blue dot, respectively. Right: a
rushed path of height~$4$. Note that progressive paths are allowed to visit
their maximal height multiple times, while rushed paths are not.}
\label{fig:defs}
\end{figure}

A rushed path can be described as a run of up steps followed by a
right-to-left \emph{culminating path}, as defined by Bousquet-Mélou and Ponty
\cite{mbm-ponty} (a culminating path is a path that visits only nonnegative
heights and visits its final height only once). This means that the
enumeration of rushed paths of a given height is easily derived from that of
culminating paths. The addition of a run of up steps can seem like a trivial
difference; however, when considering paths irrespective of height, it makes a
significant difference in the asymptotic enumeration.

A remarkable fact, due to Asinowski and Jelínek and stated in the OEIS entry
mentioned above, is that progressive and rushed paths, when counted according
to just length, have the same enumerating sequence (this is not true when
taking height into account). This calls for a bijection, but such a bijection
has not been found so far. We present such a bijection and explain how
it behaves with respect to the height statistic.

Rushed paths, as it turns out, have already been studied by Durhuus and Ünel
in the guise of \emph{one-sided trees}, defined as plane trees with a height
equal to the length of their leftmost branch \cite{durhuus2}; upon adding a
leaf on the leftmost branch, the contour path of such trees are rushed paths
(Figure~\ref{fig:trees}, right). Their enumerating sequence exhibit a
remarkable asymptotic estimate of the form~\smash{$4^ne^{-\nu
n^{1/3}}n^{-5/6}$}. This sort of \emph{stretched exponential} has attracted
recent attention in many combinatorial contexts
\cite{guttmann,chang,price,durhuus}, but it highly unusual for lattice paths,
which usually have a form~$\mu^nn^\gamma$, with~$\gamma$ normally~$0$, $-1/2$
or~$-3/2$ in one dimension~\cite{banderier}. Culminating paths, for their
part, asymptotically number~$2^{n-2}/n$~\cite[Proposition~4.1]{mbm-ponty}.
Pushed Dyck paths also have the stretched exponential, but involve weights on
paths rather than being a proper subset (see~\cite{guttmann},
or~\cite{durhuus} for their tree avatars).

\begin{figure}[ht]
\hfill%
\begin{tikzpicture}[trees]
\coordinate (a0) at (0,1);
\coordinate (a1) at (1,2);
\coordinate (a2) at (2,2);
\coordinate (a3) at (3,3);
\coordinate (a4) at (3.000000,2);
\draw[path] (a3) -- (a4);
\coordinate (a5) at (2.000000,1);
\draw[path] (a1) -- (a5);
\draw[path] (a2) -- (a5);
\draw[path] (a4) -- (a5);
\coordinate (a6) at (4,1);
\coordinate (a7) at (5,4);
\coordinate (a8) at (6,4);
\coordinate (a9) at (5.500000,3);
\draw[path] (a7) -- (a9);
\draw[path] (a8) -- (a9);
\coordinate (a10) at (5.500000,2);
\draw[path] (a9) -- (a10);
\coordinate (a11) at (7,2);
\coordinate (a12) at (6.250000,1);
\draw[path] (a10) -- (a12);
\draw[path] (a11) -- (a12);
\coordinate (a13) at (3.125000,0);
\draw[path] (a0) -- (a13);
\draw[path] (a5) -- (a13);
\draw[path] (a6) -- (a13);
\draw[path] (a12) -- (a13);
\draw[path] (a13) -- ++(0,-.5);
\end{tikzpicture}%
\hfill\hfill%
\begin{tikzpicture}[trees]
\coordinate (a0) at (0,4);
\coordinate (a1) at (0.000000,3);
\draw[path] (a0) -- (a1);
\coordinate (a2) at (1,3);
\coordinate (a3) at (0.500000,2);
\draw[path] (a1) -- (a3);
\draw[path] (a2) -- (a3);
\coordinate (a4) at (2,2);
\coordinate (a5) at (3,2);
\coordinate (a6) at (1.750000,1);
\draw[path] (a3) -- (a6);
\draw[path] (a4) -- (a6);
\draw[path] (a5) -- (a6);
\coordinate (a7) at (4,1);
\coordinate (a8) at (5,3);
\coordinate (a9) at (6,3);
\coordinate (a10) at (5.500000,2);
\draw[path] (a8) -- (a10);
\draw[path] (a9) -- (a10);
\coordinate (a11) at (7,2);
\coordinate (a12) at (6.250000,1);
\draw[path] (a10) -- (a12);
\draw[path] (a11) -- (a12);
\coordinate (a13) at (4.000000,0);
\draw[path] (a6) -- (a13);
\draw[path] (a7) -- (a13);
\draw[path] (a12) -- (a13);
\draw[path] (a13) -- ++(0,-.5);
\end{tikzpicture}%
\hfill\mbox{}%
\caption{The plane trees corresponding to the two Dyck paths of
Figure~\ref{fig:defs}. The tree on the left has the property that, for any
leaf of height~$i\ge2$, there is a leaf of height~$i-1$ to the left.}
\label{fig:trees}
\end{figure}

The paper is organized as follows. In Section~\ref{sec:bijections}, we
describe our bijection between rushed and progressive paths. In
Section~\ref{sec:enum}, we give exact and asymptotic enumeration results.
In Section~\ref{sec:random}, we give a random sampling algorithm for rushed
paths. Finally, in Section~\ref{sec:perspectives}, we give some perspectives,
including classes of directed animals linked to progressive paths.

\section{Bijections} \label{sec:bijections}

\subsection{Culminating progressive paths}

Before defining our bijection between rushed and progressive paths, we need
another bijection, which goes to Dyck paths to \emph{progressive culminating
paths} (culminating paths satisfying the same constraint as progressive
paths). Let~$P$ be a Dyck path of height~$h$, written as a word
on~$\{\U,\D\}$. Decompose it as:
\begin{equation} \label{f}
P = A_0\U\dotsm A_{h-1}\U\cdot A_{h}\cdot
\D B_{h-1}\dotsm\D B_0
\end{equation}
where the~$A_i$'s and~$B_i$'s are downward Dyck paths of height at most~$i$.
Let~$F(P)$ be the path:
\begin{equation} \label{F}
F(P) = A_0\U\cdot \D B_0\U A_1\U \dotsm\D B_{h-1}\U A_h\U\text.
\end{equation}
The decompositions of~$P$ and~$F(P)$ are illustrated in Figure~\ref{fig:F}
(the factors~$A_0$ and~$B_0$ are not shown since they are always empty).

\begin{figure}[ht]\hfill%
\begin{tikzpicture}[paths]
\foreach \y in {0,...,3} { \draw [help lines] (0,\y) -- ++(26,0); }
\draw[path] (0,0) \ustep
\dyck{-1}{color1}{$A_1$} \ustep
\dyck{-2}{color1}{$A_2$} \ustep
\dyck{-3}{color1}{$A_3$} 
\dstep \dyck{-2}{color2}{$B_2$} 
\dstep \dyck{-1}{color2}{$B_1$} 
\dstep;
\end{tikzpicture}%
\hfill\hfill%
\begin{tikzpicture}[paths]
\foreach \y in {0,...,4} { \draw [help lines] (0,\y) -- ++(30,0); }
\draw[path] (0,0) \ustep\dstep\ustep
\dyck{-1}{color1}{$A_1$} \ustep\dstep
\dyck{-1}{color2}{$B_1$} \ustep
\dyck{-2}{color1}{$A_2$} \ustep\dstep
\dyck{-2}{color2}{$B_2$} \ustep
\dyck{-3}{color1}{$A_3$} \ustep;
\end{tikzpicture}\hfill\hspace{0pt}%
\caption{Left: a Dyck path of height~$3$ decomposed as in~\eqref{f}. Right:
its image by~$F$, a progressive culminating path of height~$4$ decomposed as
in~\eqref{F}.} \label{fig:F}
\end{figure}

\begin{theorem} \label{thm:F}
The mapping~$F$ is a bijection from Dyck paths of length~$2n$ and
height~$h$ to progressive culminating paths of length~$2n+h+1$ and
height~$h+1$.
\end{theorem}

\begin{proof}
First, we check that~$F(P)$ is a culminating progressive path. By
construction, the factor~$\D B_{i-2}\U A_{i-1}u$ goes from the first visit at
height~$i-1$ to the first visit at height~$i$; it never visits negative height
and visits the height~$i-1$ at least twice, so~$F(P)$ is culminating and
progressive.

Second, the decomposition \eqref{F} can be recovered from~$F(P)$ by cutting at
the first and second visits at height~$i$, for~$i = 1,\dotsc,h-1$, which
allows us to rebuild the path~$P$. This shows that~$F$ is a bijection.
\end{proof}

\subsection{Progressive paths}

We are now ready to define our main bijection. Let~$P$ be a rushed path of
height~$h$. Decompose it as:
\begin{equation} \label{g}
P = \U^h\cdot\D C_0\dotsm\D C_{h-1}
\end{equation}
where the~$C_i$'s are Dyck paths of height at most~$i$. Let~$m$ be the maximum
height of~$C_i$ for~$i = 0,\dotsc,h-1$ and let~$j$ be the smallest index with
$h(C_j) = m$ (necessarily, we have~$j\ge m$). Let~$G(P)$ be the path:
\begin{equation} \label{G}
G(P) = F(C_j) \cdot C_0\D\dotsm C_{m-1}\D\cdot
\U C_{m}\D\dotsm\U C_{j-1}\D\cdot\D\cdot
\U C_{j+1}\D\dotsm\U C_{h-1}\D\text.
\end{equation}
The decompositions \eqref{g} and \eqref{G} are illustrated in
Figure~\ref{fig:G}.

\begin{figure}[ht]\centering%
\begin{tikzpicture}[paths]
\foreach \y in {0,...,10} { \draw [help lines] (0,\y) -- ++(56,0); }
\draw[path] (0,0) coordinate (here) 
\ustep\ustep\ustep\ustep\ustep\ustep\ustep\ustep\ustep\ustep
\dstep
\dstep \dyck{1}{color2}{$C_1$}
\dstep \dyck{2}{color2}{$C_2$}
\dstep \dyck{3}{color2}{$C_3$}
\dstep \dyck{3}{color3}{$C_4$}
\dstep \dyck{3}{color3}{$C_5$}
\dstep \dyck{4}{color1}{$C_6$}
\dstep \dyck{4}{color4}{$C_7$} 
\dstep \dyck{4}{color4}{$C_8$} 
\dstep \dyck{4}{color4}{$C_9$} ;
\end{tikzpicture}\\[2em]
\begin{tikzpicture}[paths]
\foreach \y in {0,...,5} { \draw [help lines] (0,\y) -- ++(56,0); }
\path (0,0) coordinate (here) 
\culm{color1}{$F(C_6)$} 
\dstep
\dyck{1}{color2}{$C_1$} \dstep
\dyck{2}{color2}{$C_2$} \dstep
\dyck{3}{color2}{$C_3$} \dstep
\ustep\dyck{3}{color3}{$C_4$}\dstep
\ustep\dyck{3}{color3}{$C_5$}\dstep
\dstep
\ustep \dyck{4}{color4}{$C_7$} \dstep
\ustep \dyck{4}{color4}{$C_8$} \dstep
\ustep \dyck{4}{color4}{$C_9$} \dstep ;
\end{tikzpicture}
\caption{Above: a rushed Dyck path of height~$10$ decomposed as in~\eqref{g},
where~$m = 4$ and~$j = 6$ (so~$C_6$ has height exactly~$4$, while~$C_7$, $C_8$
and~$C_9$ have height at most~$4$). Below: its image by~$G$, a progressive
Dyck path of height~$5$ decomposed as in~\eqref{G}.}
\label{fig:G}
\end{figure}

\begin{theorem} \label{thm:G}
The mapping~$G$ is a bijection from Dyck paths of length~$2n$
to progressive Dyck paths of length~$2n$.
\end{theorem}

\begin{proof}
First, we check that~$G(P)$ is a progressive path of height~$m+1$ and
length~$2n$. The factor~$F(C_k)$ is a progressive culminating path of
height~$m+1$ by Theorem~\ref{thm:F}; since the factors~$C_i$ have height at
most~$i$ for~$i = 0,\dotsc,m-1$, at most~$m-1$ for~$i = m,\dotsc,k-1$ and at
most~$m$ for~$i = k,\dotsc,h-1$, the rest of the path never visits heights
above~$m+1$. Finally, the length of~$G(P)$ is~$\lvert P\rvert + \abs{F(C_j)} -
\abs{C_j} - m - 1$, which is~$2n$ by Theorem~\ref{thm:F}.

Second, the decomposition \eqref{G} and thus the path~$P$ can be recovered
from the path~$G(P)$ by cutting at the first visit at height~$m+1$, at the
first visit at height~$i$ thereafter for~$i=m,\dotsc,2$, at every visit at
height~$1$ thereafter before the first visit at height~$0$ and at every visit
at height~$0$ thereafter. This shows that~$G$ is a bijection.
\end{proof}

\section{Enumeration and asymptotics} \label{sec:enum}

\subsection{Notations and classical results}

Let~$p_n = r_n$ be the number of progressive or rushed paths of
length~$2n$. Let~$p_{n,h}$, $r_{n,h}$ and~$d_{n,h}$ be the number of
progressive paths, rushed paths and Dyck paths, respectively, of
length~$2n$ and height~$h$. Also let~$c_{n,h}$ be the number of culminating
paths of length~$n$ and height~$h$ (culminating paths have odd length when~$h$
is odd).  Let~$R(z) = P(z)$, $R_h(z)$, $P_h(z)$, $D_h(z)$ and~$C_h(z)$ be the
generating functions of these sets, where the variable~$z$ counts the
semilength (so~$C_h(z)$ involves powers of~$z^{1/2}$ when~$h$ is odd).

We denote by~$F_h(z)$ be the Fibonacci polynomials, defined by~$F_0(z) =
F_1(z) = 1$ and~$F_h(z) = F_{h-1}(z) - zF_{h-2}(z)$ for~$h\ge2$. Also let~$q
:= q(z)$ be the generating function of the Catalan numbers, satisfying~$q =
z(1+q)^2$. We have the formula:

We now state some classical enumeration results for Dyck paths of bounded
height and culminating paths. First, we have:
\begin{equation} \label{DC}
D_h(z) = \frac{z^{h}}{F_h(z)F_{h+1}(z)}
\quad\text{and}\quad
C_h(z) = \frac{z^{h/2}}{F_h(z)}
\text,
\end{equation}
where~$F_h(z)$ are the Fibonacci polynomials defined by~$F_0(z) = F_1(z) = 1$
and~$F_h(z) = F_{h-1}(z) - zF_{h-2}(z)$ if~$h\ge2$. These polynomials also
have the following form in terms of the generating function~$q := q(z)$ of
Catalan numbers, satisfying~$q = z(1+q)^2$:
\begin{equation} \label{fibo}
F_h(z) = \frac{1 - q^{h+1}}{(1-q)(1+q)^h}
\text.
\end{equation}
We also have a closed form on the number of culminating paths of height~$h$.
Shifting the index~$h$ to get a nicer form, we have, if~$n$ and~$h-1$ have the
same parity:
\begin{equation} \label{c}
c_{n,h-1} = \frac{2^{n+1}}{h}
\sum_{j=1}^{\floor{\frac{h-1}{2}}}
(-1)^{j+1}
\sin^2\frac{j\pi}{h} \cos^{n-1}\frac{j\pi}{h}
\text.
\end{equation}
The formula for~$D_h(z)$ can be found in~\cite{bruijn}, while the one
for~$C_h(z)$ is in~\cite{mbm-ponty} (with~$q = U^2$ in Bousquet-Mélou and
Ponty's notation). The formula~\eqref{c} can be derived with a simple fraction
decomposition as explained in~\cite{bruijn}; alternatively, it is a
consequence of Laplace's formula for the gambler's ruin problem~\cite{feller}.

\subsection{Rushed paths}

As stated in Section~\ref{sec:intro}, the results of~\cite{durhuus2} on
one-sided trees translate to rushed paths. Accordingly, the results below are
essentially specializations of these results (with~$\mu = 0$ in the notation
of the paper to get the uniform measure).

\begin{theorem} \label{thm:rushed}
The generating function of rushed paths of height~$h$ is:
\begin{equation} \label{Rh}
R_h(z) = \frac{z^{h}}{F_h(z)}\text.
\end{equation}
The generating function of rushed paths is:
\begin{equation} \label{R}
R(z) = \sum_{h\ge0}
\frac{(1-q)q^h}{(1-q^{h+1})(1+q)^h}
\text.
\end{equation}
The number of rushed paths of length~$2n$ and height~$h-1$ is:
\begin{equation} \label{rh}
r_{n,h-1} = \frac{4^{n+1}}{2^hh}
\sum_{j=1}^{\floor{\frac{h-1}{2}}}
(-1)^{j+1} \sin^2\frac{j\pi}{h} \cos^{2n-h}\frac{j\pi}{h}\text.
\end{equation}
The number of rushed paths of length~$2n$ satisfies, as~$n$ tends to
infinity:
\begin{equation} \label{r}
r_n = \lambda\mkern1mu4^ne^{-\nu n^{1/3}}n^{-5/6}
\left[1 + O\left(n^{-1/3}\right)\right]
\end{equation} 
where~$\lambda = (4\pi)^{5/6}(\log2)^{1/3}/\sqrt3$ and~$\nu =
3(\pi\log2\mkern2mu/2)^{2/3}$.
\end{theorem}

\begin{proof}
The identities~\eqref{Rh}, \eqref{R} and~\eqref{rh} follow from
the fact that~$r_{n,h} = c_{2n-h,h}$. The estimate \eqref{r} is Theorem~5.1
in~\cite{durhuus2}, after accounting for the shift in size between rushed
paths and one-sided trees (the essence of the argument is to note
that~$r_{n,h}$ takes its maximum when~$h\sim(2\pi^2n/\log2)^{1/3}$ and use
Laplace's method to estimate~$r_n$).
\end{proof}

Finally, we state some useful limit law results. Let~$\mathfrak R_n$ be a
uniformly distributed rushed path of length~$2n$. If~$P$ is a Dyck path, we
call \emph{arch} of~$P$ the factor between two consecutive visits at
height~$0$ and we denote by~$a(P)$ the number of arches of~$P$.

\begin{theorem} \label{thm:rushedlaws}
We have, as~$n$ tends to infinity:
\begin{equation} \label{normal}
\frac{h(\mathfrak R_n) - \mu n^{1/3}}{\sigma n^{1/6}}\tod\mathcal N(0,1)
\end{equation}
where~$\mu = (2\pi^2/\log 2)^{1/3}$ and~$\sigma =
(2\pi^2)^{1/6}/\left((\log2)^{2/3}\sqrt3\right)$. We also have:
\begin{equation} \label{arches}
a(\mathfrak R_n)\tod\mathcal A
\quad\text{where}\quad
\mathbf E\left[u^{\mathcal A}\right] = \frac u{2-u}2\big.^{\strutfrac{u}{2-u}-1}
\text.
\end{equation}
\end{theorem}

\begin{proof}
The limit law \eqref{normal} is implied by Laplace's method in the proof of
Theorem~5.1 in \cite{durhuus2}.

To show \eqref{arches}, we first note that number of arches in the Dyck path
translates to number of children of the root vertex in the plane tree. The
information about that number is contained in the local limit of the tree
(Proposition~5.6 in \cite{durhuus2}). That limit shows that the root vertex
has~$R$ children belonging to the ``spine'', where~$R-1$ is distributed
like~$\poisson(\log2)$, and~$R$ critical Bienaymé-Galton-Watson trees grafted
in between, each accounting for~$\geom(1/2)$ more children (see the
illustration in Figure~3 in~\cite{durhuus2}). This gives the probability
generating function~\eqref{arches}.

Incidentally, this local limit gives us an intuitive picture of the arches
of~$\mathfrak R_n$: it has~$R$ large arches, looking at the limit like a
Poisson point process, and a geometric number of small arches in between each.
\end{proof}

\subsection{Progressive paths}

We conclude this section with our own contributions concerning progressive
paths. By the bijection~$G$, $P(z) = R(z)$ and~$p_n = r_n$ satisfy~\eqref{R}
and~\eqref{r}.

\begin{theorem} \label{thm:progressive}
The generating function of progressive paths of height~$h$ is:
\begin{equation} \label{Ph}
P_h(z) = \frac{z^{2h-1}}{F_{h-1}(z)F_h(z)F_{h+1}(z)}\text.
\end{equation}
Moreover, we have:
\begin{equation} \label{delta}
h(\mathfrak R_n) - h\bigl(G(\mathfrak R_n)\bigr)
\tod\mathcal D\quad\text{where}\quad
\mathbf E\left[u^{\mathcal D}\right] = \frac1u\left(1 +
\frac{4(u-1)}{2-u}2\big.^{\strutfrac{u}{2-u} - 1}\right)
\text.
\end{equation}
In particular, the height of a random progressive path of length~$2n$
also has the limit law~\eqref{normal}.
\end{theorem}

\begin{proof}
Progressive paths of height~$h$ are the product of a progressive culminating
path of height~$h$ (generating function~$z^{h/2}D_{h-1}(z)$ by
Theorem~\ref{thm:F}) and a downward culminating path of height~$h+1$, minus
the last downstep (generating function~$C_{h+1}(z)/z^{1/2}$). This
gives~\eqref{Ph}.

It remains to prove that~$\Delta_n := h(\mathfrak R_n) - h(G(\mathfrak R_n))$
tends to the limit law~$\mathcal D$. If~$h(\mathfrak R_n) = h$,
decompose~$\mathfrak R_n$ as in~\eqref{g}; according to Theorem~\ref{thm:G},
we have~$\Delta_n \ge k$ if and only if all the~$A_i$'s have height at
most~$h-k-1$. If this is the case, consider the path:
\[\U^{h-k+1}\cdot\D C_0\dotsm\D C_{h-k-1}\cdot\D \cdot\U C_{h-k}\D\dotsm\U
C_{h-1}\D\text.\]
This is a rushed path of length~$2n+2$ and height~$h-k+1$ with~$a+1$ arches,
and the transformation is bijective. Denoting by~$r_{n,h,a}$ the number of
rushed paths with length~$2n$, height~$h$ and~$a$~arches, we have
therefore:
\[\mathbf P\bigl(\Delta_n \ge k \mathbin{\big\vert} h(\mathfrak R_n) = h\bigr)
= \frac{r_{n+1,h-k+1,k+1}}{r_{n,h}}\text.\]
Fix~$k$ and let~$n$ tend to infinity. With~\eqref{rh}, expanding the sine and
cosine, we get:
\[r_{n,h} \sim
\frac{\pi^24^{n+1}}{2^hh^3}e^{-\strutfrac{\pi^2n}{h^2}}\text.\]
Since~$h\sim\mu n^{1/3}$ with high probability from~\eqref{normal}, this shows
that we have the form~$r_{n,h} \sim \kappa 4^n\smash{e^{-\nu n^{1/3}}n^{-1}}$
for some constant~$\kappa$ and therefore that the
ratio~$r_{n+1,h-k+1}/r_{n,h}$ tends to~$4$ with high probability. Therefore,
we have from~\eqref{arches}:
\[\mathbf P(\Delta_n \ge k) \to
4\mspace{1mu}\mathbf P(\mathcal A = k+1)\text.\]
This shows convergence in distribution to~$\mathcal D$.
\end{proof}

\section{Random sampling} \label{sec:random}

In this section, we consider the problem of finding an algorithm returning a
uniformly distributed rushed path of length~$2n$ (a random progressive path
can then simply be obtained using the bijection~$G$).
The best algorithms for sampling lattice paths are usually based on
anticipated rejection: paths are drawn randomly one step at a time and those
paths that do not satisfy the wanted constraints, like staying above the
$x$-axis, are rejected during the generation. This is the idea behind the
Florentine algorithms \cite{florentine,florentine2} and the algorithms in
\cite[Section~6.2]{mbm-ponty}. In favorable cases, such algorithms often
achieve linear complexity in the size of the output \cite{rejection}.

In our case, however, using such a scheme seems difficult. If we sample a
rushed path of length~$2n$ and height~$h$ by outputting~$h$ up steps
followed by a random path, the probability of sampling any given path without
rejecting is~$2^{-2n+h}$, but we need a probability independent of~$h$ to
ensure uniformity. Moreover, a random path of length~$2n$ will have a
height~$O(\sqrt n)$, but we need it to be~$O(n^{1/3})$. This seems to show
that the rejection rate will be prohibitive.

We instead present an algorithm running in two stages: first, we sample the
height~$h$ of the path with the appropriate random distribution. Second, we
draw a uniformly distributed path of length~$2n$ and height~$h$. We present
both stages below, starting with the second.

\subsection{Rushed paths of fixed height}

Let~$\mathcal D_{h,i,j,\ell}$ be the set of paths of length~$\ell$, starting
at height~$i$, ending at~$j$ and staying at heights from~$0$ to~$h-1$. By
considering the height at the middle point of the path, we have
when~$\ell\ge1$:
\begin{equation} \label{D}
\mathcal D_{h,i,j,\ell} = \bigcup_{k=0}^{h-1}
\mathcal D_{h,i,k,\floor{\frac\ell2}}
\mathcal D_{h,k,j,\ceil{\frac\ell2}}\text.
\end{equation}
Let~$d_{h,i,j,\ell} = \#\mathcal D_{h,i,j,\ell}$. Provided we can
compute~$d_{h,i,j,\ell}$, we can devise an \emph{unranking} procedure: given
an integer~$a$, we get the~\ath\ path of~$\mathcal D_{n,i,j,\ell}$ for the
ranking implied by~\eqref{D} in a divide-and conquer fashion
(Algorithm~\ref{alg:rh}). Thus, we can get the~\ath\ element of~$\mathcal
R_{n,h} = \U^h\D\cdot\mathcal D_{h,h-1,0,2n-h-1}$.

\begin{algorithm}[ht]
\caption{$a$\nobreakdashes-th path of~$\mathcal D_{h,i,j,\ell}$} \label{alg:rh}
\begin{algorithmic}[1]
\If{$n = 1$}
 \If{$j = i+1$}
  \State \Return $\U$
 \ElsIf{$j = i-1$}
  \State \Return $\D$
 \EndIf
\Else
 \For{$k = 0,\dotsc,h-1$ \textbf{if} $k \equiv \floor{\ell/2} + i \bmod 2$}
  \State $c_1\gets d_{h,i,k,\floor{\frac\ell2}}$
  \State $c_2\gets d_{h,k,j,\ceil{\frac\ell2}}$
  \If{$a < c_1c_2$}
   \State $P_1\gets(a\bmod c_1)$-th path of $\mathcal D_{h,i,k,\floor{\ell/2}}$
   \State $P_2\gets\floor{a/c_1}$-th path of $\mathcal D_{h,k,j,\ceil{\ell/2}}$
   \State \Return $P_1P_2$
  \EndIf
  \State $a\gets a - c_1c_2$
 \EndFor
\EndIf
\end{algorithmic}
\end{algorithm}

Following the recursive method~\cite{wilf}, we could use~\eqref{D} to
precompute the numbers~$d_{h,i,j,\ell}$, but this would take
time~$O\bigl(M(n)h^3\bigr)$ (we denote by~$M(n)$ the complexity of the
multiplication of~$n$-bit integers) and space~$O(nh^2)$. We can do better by
using the following form~\cite[Theorem~10.11.1]{krattenthaler} for~$i\le j$:
\begin{equation} \label{d}
d_{h,i,j,\ell} = \left[z^{\strutfrac{\ell-j+i}{2}}\right]\frac{F_i(z)F_{h-1-j}(z)}{F_h(z)}
\text.
\end{equation}
This can be computed using Bostan and Mori's algorithm~\cite{bostan}. This
algorithm works by computing a \emph{slice} of~$h$ consecutive coefficients
of~$1/F_h(z)$ using~$\log\ell$ multiplications of polynomials of degree~$h$.
From this slice, one can compute~$d_{h,i,j,\ell}$ using~$h$ integer
multiplications. We can compute the slice only once for every value of~$\ell$
needed.

\begin{theorem} \label{thm:randD}
Algorithm~\ref{alg:rh} outputs the~$a$th path of~$\mathcal R_{n,h}$ in
time~$O\bigl(nhM(h)\bigr)$.
\end{theorem}

\begin{proof}
The correction of the algorithm is a direct consequence of~\eqref{D}. To prove
the complexity, we first note that computing a slice of coefficients
of~$1/F_h(z)$ for the length~$\ell = 2n-h-1$ costs~$O\bigl(M(nh)\log n\bigr)$;
the costs for the lengths~$\ell/2$, $\ell/4$, etc.\ are negligible as the size
of the integers is halved each step. Since~$h < n$ and~$M(n)$ is subquadratic,
this is smaller than~$O\bigl(nhM(h)\bigr)$.

Excluding the recursive calls, computing the values~$c_1$ and~$c_2$
takes, in the worst case, $h^2$ multiplications of an~$n$\nobreakdash-bit by
an~$h$\nobreakdash-bit integer; as such a multiplication
costs~$O\bigl(nM(h)/h\bigr)$, we the result by the Master Theorem.
\end{proof}

\subsection{General rushed paths}

To sample a path of~$\mathcal R_n$, we need to sample a height with the
correct distribution. This can be done by simply computing the values~$r_n$
and~$r_{n,h}$ for~$1\le h\le n$, but this is much too expensive. Instead, we
use a method in the spirit of~\cite{denise}, relying on intervals computed and
refined on demand using floating-point arithmetic. The method is conceptually
simpler than~\cite{denise} and the more recent work~\cite{bodini}, as no
branching is involved and we are considering the numbers~$r_{n,h}$
individually and not their generating function. It could function just as well
in similar contexts.

In the following, we shorten the notation to~$r := r_n$ and~$r_h := r_{n,h}$.
We consider intervals~$[r_h^-,r_h^+]$ containting~$r_h$, with a default
interval~$[0,2^{2n-h-1}]$ (we say that the interval is \emph{uninitialized} if
it still has this value). This determines an upper bound~$r^+$ for~$r$, equal
to the sum of the~$r_h^+$. Assume that we are able to compute an initial
interval~$[r_h^-,r_h^+]$ and refine it on demand, so that the interval
ultimately comes to a single point~$r_h$. Then we can sample a uniform rushed
path using Algorithm~\ref{alg:rushed}.

\begin{algorithm}[ht]%
\caption{Uniform random rushed path of length~$2n$} \label{alg:rushed}
\begin{algorithmic}[1]
\Require an array of intervals~$[r_h^-,r_h^+]$, initially empty
\Require a bound~$r^+$, initially~$2^{2n-1} - 2^{n-1}$
\State $b\gets$ random integer between~$0$ and~$r^+-1$
\State $a\gets b$
\For{$h = 1,\dotsc,n$}
 \If{$[r^-_h,r^+_h]$ is uninitialized}
  \State compute the initial interval $[r^-_h,r^+_h]$
  \State $r^+\gets r^+ - 2^{2n-h-1} + r^+_h$
  \If{$b\ge r^+$}
   \State start over
  \EndIf
 \EndIf
 \If{$a < r^+_h$}
  \While{$a\in [r^-_h,r^+_h)$}
   \State $r\gets r - r^+_h$
   \State refine the interval $[r^-_h,r^+_h]$
   \State $r\gets r + r^+_h$
  \EndWhile
  \If{$a \ge r^+_{h}$}
   \State start over
  \Else
   \State $P\gets a$-th path of $\mathcal D_{h,h-1,0,n-h-1}$
   \State \Return $\U^h\D\cdot P$
  \EndIf
 \EndIf
 \State $a\gets a - r^+_h$
\EndFor
\end{algorithmic}
\end{algorithm}

To keep things simple, one could do all the computations in exact values
instead of intervals. In this case, lines~10 to~16 are never used.

To do the computations with intervals, we used the Arb module~\cite{johansson}
of the FLINT library~\cite{flint}, which provides for ball arithmetic
(floating-point values with a rigorous bound on the error) with arbitrary
precision. We set a size for the mantissa (e.g., 32 bits), run Bostan and
Mori's algorithm, and hope that the relative radius of the resulting ball
remains small. To refine the interval, we double the width of the mantissa and
run the computation again.

Unfortunately, the fast algorithms for polynomial multiplication (e.g.,
Karatsuba's or Schönage--Strassen's) are not numerically stable, so the radius
explodes and we do not get anything useful.  We used the naïve multiplication
instead, but ways to make it work with fast multiplication can be found
in~\cite{vanderhoeven}. Experimentally, it seems that in order to get~$p$ bits
of precision in the final balls, we need to use a mantissa of~$p+O(\log n)$
bits.

\begin{theorem} \label{thm:randR}
Algorithm~\ref{alg:rushed} outputs a uniformly distributed path of~$\mathcal
R_n$. If we only do exact computations, picking a height takes
time~$O\bigl(n^{1/3}M(n^{4/3})\log n\bigr)$ with high probability. With ball
arithmetic, if computing~$p$ bits of precision requires~$p+O(\log n)$ bits of
mantissa, the complexity is~$O(n\log^3n)$.
\end{theorem}

With exact computations, picking a height is slower than unranking the path
with Algorithm~\ref{alg:rh}, which takes~$O\bigl(n^{4/3}M(n^{1/3})\bigr)$.
With the Schönage--Strassen algorithm, we still just achieve subquadratic
complexity. With intervals, however, the cost of picking a height becomes
negligible.

\begin{proof}
First, every time the algorithm starts over, it either initialized or refined
an interval. Since the intervals ultimately come to a point, this means that
the algorithm always terminates.

To show that the output is uniform, we note that, for any rushed path~$P$,
there is exactly one value of~$b$ that entails outputting~$P$ without
starting over; namely, if~$P$ is the~\ath\ path of~$\mathcal R_{n,h}$, the
number~$b$ should be~$r_1^+ + \dotsb + r_{h-1}^+ + a$, where the~$r_i^+$ are
the existing bounds if they are initialized and the initial bounds if not.
This shows uniformity.

To show the complexity, we first estimate how many of the terms~$r_h$ need to
be computed. If~$m$ terms have been initialized, the bound~$r^+$ is, at most,
$r + 2^{2n-m-1}$. According to~\eqref{r}, this shows that~$r^+$ becomes very
close to~$r$ when~$m = O(n^{1/3})$. Thus, with high probability, we only need
to initialize~$O(n^{1/3})$ terms.

Using Bostan and Mori's algorithm, computing~$r_h$ takes~$O(\log n)$
multiplications of polynomials of degree~$O(n^{1/3})$. If we do exact
computation, the coefficients have~$O(n)$ bits, which
gives~$O\bigl(n^{1/3}M(n^{4/3})\log n\bigr)$. If we do ball arithmetic,
having~$O(\log n)$ bits of precision on~$O(n^{1/3})$ coefficients is enough to
ensure that the probability of starting over tends to zero. Thus, we need
only~$O(\log n)$ bits of mantissa. With the naïve multiplication, the
complexity is~$O(n\log^3n)$.
\end{proof}

Sample rushed paths and their progressive images are shown in
Figure~\ref{fig:random}.

\begin{figure}[ht]\hfill%
\renewcommand\ustep{ -- ++(.01,.01) coordinate (here)}%
\renewcommand\dstep{ -- ++(.01,-.01) coordinate (here)}%
\tikzstyle{small paths}=[scale=.28,yscale=10,line width=.05pt,baseline=0cm]%
\begin{tikzpicture}[small paths]
\input{1r.tex}
\begin{scope}[yshift=-.5cm]
\input{2r.tex}
\end{scope}
\begin{scope}[yshift=-1cm]
\input{3r.tex}
\end{scope}
\end{tikzpicture}%
\hfill\hfill%
\begin{tikzpicture}[small paths]
\input{1p.tex}
\begin{scope}[yshift=-.5cm]
\input{2p.tex}
\end{scope}
\begin{scope}[yshift=-1cm]
\input{3p.tex}
\end{scope}
\end{tikzpicture}\hfill\hspace{0pt}%
\caption{Left: three random rushed paths of length~$2000$. Right: the
images of these paths by the bijection~$G$. Colors are the same as in
Figure~\ref{fig:G}.} \label{fig:random}
\end{figure}

\section{Perspectives} \label{sec:perspectives}

Many interesting problems remain on this topic. First, classes of Dyck paths
similar to progressive paths may exhibit similar asymptotical behavior: OEIS
entry \oeis{A287776} describes paths one could name \emph{doubly
progressive}---both left-to-right and right-to-left progressive. The
definition of progressive paths also naturally extends to unconstrained paths:
a path is progressive if, before its first visit at every
height~$i\not\in\{0,1\}$, it visits twice either~$i-1$ or~$i+1$. There are
many possible variations and it would be interesting to see which can be
enumerated and if results like~\eqref{r} and~\eqref{normal} hold.

Dyck paths are also linked to directed animals on the triangular lattice, via
bijections with heaps of dimers~\cite{viennot,rechnitzer}. Adopting the
convention that animals grow towards the right side, we say that a directed
animal is \emph{acute} if, for every site at ordinate~$y\ne0$, there are at
least two sites to the left with ordinate~$y-1$ or~$y+1$. Progressive paths
are in bijection with acute half-animals (all sites above the~$x$-axis, see
Figure~\ref{fig:anim}). Full acute animals seem harder to enumerate.

\begin{figure}[ht]
\hfill%
\begin{tikzpicture}[paths]
\foreach \y in {0,...,4} { \draw[help lines] (0,\y) -- ++(30,0); }
\draw[path] (0,0) \walk uduududuuddduduuuudduddduduudd;;
\end{tikzpicture}%
\hfill\hfill%
\begin{tikzpicture}[triang]
\begin{scope}
\clip [xslant=-.5] (-.1,-.1) rectangle (8,3.1);
\draw[help lines] (0,0) grid (8,3);
\foreach \x in {0,...,3} { \draw[help lines] (\x,0) -- (0,\x); }
\foreach \x in {4,...,10} { \draw[help lines] (\x,0) -- ++(-3,3); }
\end{scope}
\foreach \p in
{(0,0),(1,0),(1,1),(2,1),(3,1),(3,2),(4,0),(5,0),(5,1),(5,2),(5,3),(6,0),(6,2),(7,0),(7,1)}
{ \node [site,minimum size=3.5pt] at \p {}; }
\end{tikzpicture}%
\hfill\mbox{}%
\caption{Left: a progressive path. Right: the corresponding acute half-animal:
for every site not on the bottom row, there are at least two sites to the left
in the row just below.} \label{fig:anim}
\end{figure}

Finally, one may want to go beyond Dyck paths, to Motzkin paths, Schröder
paths, $m$-Dyck paths, or more general models~\cite{banderier,mbm}, as well as
directed animals on different lattices. It is not clear, however, which of the
many possibilities is the ``right'' extension of the definition of progressive
paths. The adaptation of the asymptotic results may also prove challenging.

\bibliographystyle{abbrv}
\bibliography{biblio}{}

\end{document}